\documentclass[sigconf,screen=true,natbib=true]{acmart}
\pdfoutput=1

\usepackage{graphicx}
\usepackage[english]{babel}
\usepackage[toc,page]{appendix}
\usepackage{hyperref}
\usepackage{graphicx}
\usepackage{textcomp}
\usepackage{xcolor}
\usepackage{mathtools}
\usepackage{enumerate, url}
\usepackage[normalem]{ulem}
\usepackage{caption}
\usepackage{algorithm}
\usepackage{algpseudocode}
\usepackage{mleftright}
\usepackage{tikz-cd}
\usepackage[normalem]{ulem}
\usepackage{hyperref}
\usepackage{pgf}
\usepackage{kbordermatrix}
\usepackage{multirow}
\usepackage{tikz}
\usetikzlibrary{positioning,fit,arrows.meta,backgrounds}
\usetikzlibrary{arrows,automata}
\usepackage{booktabs}
\usepackage{mathtools}
\usepackage{color}
\usepackage{graphicx}
\usepackage{pgfplots}
\usepgfplotslibrary{fillbetween}
\pgfplotsset{compat=newest}
\usepackage{thmtools}
\usepackage{thm-restate}

\usepackage{scalerel,stackengine}
\stackMath
\newcommand\reallywidehat[1]{%
\savestack{\tmpbox}{\stretchto{%
  \scaleto{%
    \scalerel*[\widthof{\ensuremath{#1}}]{\kern-.6pt\bigwedge\kern-.6pt}%
    {\rule[-\textheight/2]{1ex}{\textheight}}
  }{\textheight}%
}{0.5ex}}%
\stackon[1pt]{#1}{\tmpbox}%
}
\usepackage{soul}
\usetikzlibrary{calc}
\makeatletter
\newif\if@anonymize
\@anonymizetrue    

\if@anonymize
  \newcommand{\highlight@DoHighlight}{
    \fill [outer sep = -15pt, inner sep = 0pt, color=black]
          ($(begin highlight)+(0,8pt)$) rectangle ($(end highlight)+(0,-3pt)$) ;
  }

  \newcommand{\highlight@BeginHighlight}{
    \coordinate (begin highlight) at (0,0) ;
  }

  \newcommand{\highlight@EndHighlight}{
    \coordinate (end highlight) at (0,0) ;
  }

  \newdimen\highlight@previous
  \newdimen\highlight@current
  \newlength{\item@width}

  \DeclareRobustCommand*\anonymize{%
    \SOUL@setup
    \def\SOUL@preamble{%
      \begin{tikzpicture}[overlay, remember picture]
        \highlight@BeginHighlight
        \highlight@EndHighlight
      \end{tikzpicture}%
    }%
    \def\SOUL@postamble{%
      \begin{tikzpicture}[overlay, remember picture]
        \highlight@EndHighlight
        \highlight@DoHighlight
      \end{tikzpicture}%
    }%
    \def\SOUL@everyhyphen{%
      \discretionary{%
        \SOUL@setkern\SOUL@hyphkern
        \SOUL@sethyphenchar
        \tikz[overlay, remember picture] \highlight@EndHighlight ;%
      }{%
      }{%
        \SOUL@setkern\SOUL@charkern
      }%
    }%
    \def\SOUL@everyexhyphen##1{%
      \SOUL@setkern\SOUL@hyphkern
      \settowidth{\item@width}{##1}%
      \makebox[\item@width]{}%
      \discretionary{%
        \tikz[overlay, remember picture] \highlight@EndHighlight ;%
      }{%
      }{%
        \SOUL@setkern\SOUL@charkern
      }%
    }%
    \def\SOUL@everysyllable{%
      \begin{tikzpicture}[overlay, remember picture]
        \path let \p0 = (begin highlight), \p1 = (0,0) in \pgfextra
          \global\highlight@previous=\y0
          \global\highlight@current =\y1
        \endpgfextra (0,0) ;
        \ifdim\highlight@current < \highlight@previous
          \highlight@DoHighlight
          \highlight@BeginHighlight
        \fi
      \end{tikzpicture}%
      \settowidth{\item@width}{\the\SOUL@syllable}%
      \makebox[\item@width]{}%
      \tikz[overlay, remember picture] \highlight@EndHighlight ;%
    }%
    \SOUL@
  }
\else
  \newcommand{\anonymize}[1]{#1}
\fi
\makeatother
\newtheorem{theorem}{Theorem}[section]
\newtheorem{proposition}[theorem]{Proposition}
\newtheorem{proposition/definition}[theorem]{Proposition/Definition}
\newtheorem{lemma}[theorem]{Lemma}
\newtheorem{corollary}[theorem]{Corollary}

\newtheorem{example}[theorem]{Example}

\declaretheoremstyle[
  spaceabove=0pt,
  spacebelow=0pt,
  bodyfont=\normalfont,
  postheadspace=0em,
  qed=\qedsymbol,
  headpunct={},
  headformat={}
]{withouthead}

\newcommand{\tp}{{\scriptscriptstyle\mathsf{T}}}

\let\O\undefined

\DeclareMathOperator{\O}{O}

\DeclareMathOperator{\Id}{Id}

\let\supp\undefined%
\DeclareMathOperator*{\supp}{supp}
\DeclareFontFamily{U} {MnSymbolC}{}
\DeclareFontShape{U}{MnSymbolC}{m}{n}{
  <-6> MnSymbolC5
  <6-7> MnSymbolC6
  <7-8> MnSymbolC7
  <8-9> MnSymbolC8
  <9-10> MnSymbolC9
  <10-12> MnSymbolC10
  <12-> MnSymbolC12}{}
\DeclareFontShape{U}{MnSymbolC}{b}{n}{
  <-6> MnSymbolC-Bold5
  <6-7> MnSymbolC-Bold6
  <7-8> MnSymbolC-Bold7
  <8-9> MnSymbolC-Bold8
  <9-10> MnSymbolC-Bold9
  <10-12> MnSymbolC-Bold10
  <12-> MnSymbolC-Bold12}{}
\DeclareSymbolFont{MnSyC} {U} {MnSymbolC}{m}{n}
\DeclareMathSymbol{\plus}{\mathrel}{MnSyC}{20}
\newcommand{\p}{\plus}

\newcommand{\Rmnum}[1]{\expandafter\@slowromancap\Romannumeral #1@}
\tikzset{
    module/.style={%
        draw, rounded corners,
        minimum width=#1,
        minimum height=7mm,
        font=\sffamily
        },
    module/.default=2cm,
    >=LaTeX
}

\title{Lower Bounds of Functions on Finite Abelian Groups}

\author{Jiangting Yang}
\affiliation{
\institution{Key Lab of Mathematics Mechanization, AMSS }
\department{ University of Chinese Academy of Sciences}
\city{Beijing, 100190}
\country{China}
}
\email{yangjianting@amss.ac.cn}

\author{Ke Ye}
\affiliation{
\institution{Key Lab of Mathematics Mechanization, AMSS}
\department{ University of Chinese Academy of Sciences}
\city{Beijing,100190}
\country{China}
}
\email{keyk@amss.ac.cn}

\author{Lihong Zhi}
\affiliation{
\institution{Key Lab of Mathematics Mechanization, AMSS }
\department{ University of Chinese Academy of Sciences}
\city{Beijing, 100190}
\country{China}
}
\email{lzhi@mmrc.iss.ac.cn}


\begin{document}

\begin{abstract}
The problem of computing the optimum  of a function on  a finite set is  an important problem  in mathematics and computer science. Many combinatorial problems such as MAX-SAT and MAXCUT can be recognized as optimization problems on the hypercube $C_2^n = \{-1,1\}^n$ consisting of $2^n$ points. It has been noticed that if a finite set is equipped with an abelian group structure, then one can efficiently certify nonnegative functions on it by Fourier sum of squares (FSOS). Motivated by these works, this paper is devoted to developing a framework to find a  lower bound of a function on a finite abelian group efficiently. We implement our algorithm by the  SDP solver SDPNAL+  for computing the verified lower bound of $f$ on $G$  and test it on  the  MAX-2SAT and
 MAX-3SAT benchmark problems from the Max-SAT competitions in 2009 and 2016. Beyond that, we also test our algorithm on random functions on  $C_3^n$.  These experiments demonstrate the advantage of our algorithm over previously known methods.
\end{abstract}
\maketitle

\section{Introduction}\label{sec:intro}
The problem of computing the optimum of a function on a finite set is an important but challenging problem in mathematics and computer science. Examples include the Knapsack problem\cite{Mathews1896,martello1990knapsack}, the set cover problem\cite{korte2011combinatorial,MR3238990}, the travelling salesman problem\cite{robinson1949hamiltonian,MR4398822}, the vehicle routing problem\cite{dantzig1959truck,toth2002vehicle}, the $k$-SAT problem \cite{cook1971complexity,karp1972reducibility} and their numerous variants. Although each of these problems can be formulated as an integer programming, there do not exist polynomial time algorithms for most of them unless \textbf{P}$=$\textbf{NP} \cite{karp1972reducibility,krentel1988complexity, papadimitriou1995computational, korte2011combinatorial,MR455550,toth2002vehicle}. Therefore various approximation algorithms are employed to resolve the issue \cite{vazirani2001approximation,MR1427541,goemans1995improved,VVH2008,christofides1976worst,laurent2003comparison}. Semidefinite programming (SDP) is one of the most powerful and extensively studied technique to design and analyze an approximation algorithm\cite{MR2582901,MR2306295,646129,VVH2008, MR2644353, austrin2007balanced,377033,kurpisz_et_al:LIPIcs:2016:6336,zhang2021sparse,slot2022sum,lasserre2016max}. Among those successful applications of the SDP technique, the most well-known ones are MAX-2SAT\cite{MR1412228}, MAX-3SAT\cite{646129} and MAX-CUT\cite{MR2306295}. On the other side, it is noticed in \cite{fawzi2016sparse,sakaue2017exact,yang2022computing, yang2022short} that if a finite set is equipped with an abelian group structure, then one can efficiently certify nonnegative functions on it by Fourier sum of squares (FSOS). Motivated by these works, this paper is concerned with establishing a framework to solve the following problem by FSOS and SDP.
\begin{restatable}[lower bound by FSOS]{problem}{bound}\label{problem:lower bound}
Given a function $f$ on a finite abelian group, find a lower bound of $f$ efficiently.
\end{restatable}
Let $S\subseteq \mathbb{C}^n$ be an algebraic variety. Algebraically, identifying $\mathbb{C}[S]$ with $\mathbb{C}[z_1,\dots, z_n]/I(S)$ is a favourable perspective as the latter ring is endowed with rich geometric and algebraic structures. For computational purposes, however, regarding a function as an equivalence class is not convenient, on account of the fact that an equivalence can be represented by two different polynomials. If $S = G$ is a finite group, then there is an alternative algebraic structure on $\mathbb{C}[G]$ which is extremely useful for computations \cite{fawzi2016sparse,sakaue2017exact,yang2022computing, yang2022short}. Namely, one can identify $\mathbb{C}[G]$ with the group ring of $G$ via the Fourier transform \cite{fulton2013representation}. The advantage of such a point of view is that a function $f$ on $G$ can be expanded as $f = \sum_{\chi\in \widehat{G}} \widehat{f}(\chi) \chi$, where $\widehat{G}$ is the dual group of $G$ and $\widehat{f}(\chi)$ is the Fourier coefficient of $f$ at $\chi\in \widehat{G}$. This well-known viewpoint enables us to introduce analytic tools to solve Problem~\ref{problem:lower bound}.
\subsection*{related works and our contributions}
Recently, a method for general-purpose polynomial optimization called the TSSOS hierarchy is proposed in \cite{WML21}. The new method follows the well-known methodology established in \cite{Lasserre01}, but it exploits the sparsity of polynomials to reduce the size of SDP. Combing the TSSOS hierarchy and the method in \cite{WKKM06} for correlative sparsity, \cite{WMLM22} introduces the CS-TSSOS hierarchy for large scale sparse polynomial optimization. In particular, both TSSOS and  CS-TSSOS hierarchies are applicable to optimization problems on finite abelian groups.

Many combinatorial problems, such as MAX-SAT and MAX-CUT, can be recognized as optimization problems on the hypercube $C_2^n = \{-1,1\}^n$. Due to its great importance in computer science, there are various solvers for MAX-SAT. For instance, in \cite{VVH2008}, the quotient structure of $\mathbb{C}[C_2^n]$ is explored for support selection strategies, from which one can solve MAX-SAT by SDP; by combing several optimization techniques specifically designed for MAX-SAT, \cite{wang2019low} provides an efficient SDP-based MIXSAT algorithm; based on the resolution refutation, a solver called MS-builder is proposed in \cite{py2021proof}.

On the one hand, TSSOS and CS-TSSOS hierarchies can handle general polynomial optimization problems, while specially designed solvers such as MIXSAT and MS-builder can only deal with MAX-SAT problems. On the other hand, however, it is natural to expect that these specially designed solvers would outperform general-purpose methods on MAX-SAT problems.

It is our framework of optimization by FSOS that balances the universality and efficiency. Indeed, our method is applicable to any optimization problems on finite abelian groups, including the hypercube $C_2^n$, cyclic group $\mathbb{Z}_N$ and their product.

The main contribution of this paper is the approximation algorithm for the minimum of a function $f$ on an abelian group $G$. To be more specific, for each positive integer $k$, we propose an SDP (Lemma~\ref{lem:sdp formulation-1}) associated with some $S_f\subseteq \widehat{G}$ (Algorithm~\ref{alg:selection}) where $|S_f| = k$. A solution to the SDP provides a good lower bound for $f$ (Theorem~\ref{thm:SOS_with_error}). Since the size of the SDP is $k \times k$, our approximation algorithm is applicable to functions on a large abelian group. Another feature of our method is the support selection strategy for the SDP. Unlike the strategy suggested in \cite{VVH2008}, which is almost the same for all functions on the hypercube, our strategy provides different bases for different functions by exploiting the magnitudes of their Fourier coefficients. Moreover, we provide two rounding methods  based on the  nullspace of the Gram matrix and the lower rank approximation of the moment matrix, respectively. We test our algorithm (Algorithm~\ref{alg:Fast_Lower_Bounds}) and rounding methods on numerical examples in Appendix~\ref{append:experiments}, from which one may clearly see the advantage of our algorithm over aforementioned methods.
\subsection*{applications}
Below we mention some potential applications of our algorithm. Although we can test our algorithm on all these interesting problems and present numerical experiments in Appendix~\ref{append:experiments},  we only concentrate on MAX-2SAT and MAX-3SAT, due to the page limit.
\subsubsection*{SAT problems}
The characteristic function (cf. Subsection~\ref{subsec:characteristic function}) of a Boolean formula in a $k$ Conjunctive Normal Form ($k$-CNF\footnote{A $k$-CNF formula is a CNF formula in which each clause contains at most $k$ literals.}) with $n$ variables can be recognized as a non-negative function on the group $C_2^n$. As a consequence, our framework provides an approximation algorithm for MAX-$k$SAT\footnote{Here we adopt the most commonly used definition \cite{krentel1988complexity} of MAX-$k$SAT, which requires one to compute the maximum number of simultaneously satisfiable clauses in a $k$-CNF formula. There are other definitions of MAX-$k$SAT, though. For instance, it is defined in \cite{arora2009computational} to be the \textbf{NP}-hard problem of finding a maximal satisfiable assignment. In \cite{papadimitriou1995computational}, MAX-$k$SAT means the \textbf{NP}-complete problem which determines the existence of an assignment satisfying at least a given number of clauses.} and its variants such as UN-SAT and MIN-SAT.
\subsubsection*{MAX-CUT problems}
Let $G = (V,E)$ be an undirected graph where $|V| = n$ with edge weights $\{w_{ij}:\{i,j\}\in E\}$. The MAX-CUT problem for $G$ is a partition $V = S \bigsqcup T$ such that the sum of weights of edges between $S$ and $T$ is as large as possible. Thus it can be formulated as an optimization problem on $\mathbb{Z}_2^n$.
\begin{align*}
&\max \sum_{\{i,j\}\in E} w_{ij} x_i (1 - x_j) \\
&s.t.~~x_i \in \mathbb{Z}_2 = \{0,1\}, 1 \le i \le n
\end{align*}
\subsubsection*{ground states of a lattice model}
A lattice model is a fixed set of locations in $\mathbb{R}^3$ on which interacting particles are distributed. A ground state of a lattice model is a configuration of these particles which minimizes the Hamiltonian. Mathematically, computing the ground state of a lattice model may be rephrased as an optimization problem on $\mathbb{Z}_{d_1} \times \cdots \times \mathbb{Z}_{d_n}$ where $n$ is the number of particles and $d_j$ is determined by the type of the $j$-th particle. In \cite{HKDRUCLC16}, the case where $d_1 = \cdots = d_n = 2$ is considered.
\section{Preliminaries}\label{sec:pre}
In this section, we first review the Fourier analysis on abelian groups. After that we provide a brief introduction to Fourier sum of squares (FSOS) on abelian groups. Lastly, we define characteristic functions of (weighted) $k$-CNF formulae, which are important examples of integer-valued functions on abelian groups.
\subsection{Fourier analysis on groups}\label{subsec:Fourier analysis on groups}
We briefly summarize fundamentals of group theory and representation theory in this subsection. For more details, we refer interested readers to \cite{rudin1962fourier,fulton2013representation,MR3443800}.

Let $G$ be a finite abelian group. A \emph{character} of $G$ is a group homomorphism $\chi: G \to \mathbb{C}^{\times}$. Here $\mathbb{C}^{\times}$ is $\mathbb{C} \setminus \{0\}$ endowed with the multiplication of complex numbers as the group operation. The set of all characters of $G$ is denoted by $\widehat{G}$, called the \emph{dual group}\footnote{Since $G$ is an abelian group, $\widehat{G}$ is indeed a group and $\widehat{G} \simeq G$.} of $G$.

According to \cite[Chapter~1]{fulton2013representation}, any function $f$ on $G$ admits the \emph{Fourier expansion}:
\[
f =\sum_{\chi \in \widehat{G}}  \widehat{f} (\chi)\chi,
\]
where $ \widehat{f} (\chi) \coloneqq \frac{1}{|G|} \sum_{g\in G} f(g) \overline{\chi(g)}$ is called the \emph{Fourier coefficient} of $f$ at $\chi \in \widehat{G}$. The \emph{support} of $f$ is
\[
\operatorname{supp}(f) \coloneqq \left\lbrace \chi\in \widehat{G}: \widehat{f} (\chi)\neq 0 \right\rbrace.
\]

As an example, the dual group of the hypercube $C_2^n = \{-1,1\}^n$ is
\[
\widehat{C_2^n} = \lbrace
z^\alpha: \alpha = (\alpha_1,\dots, \alpha_n)\in \mathbb{Z}_2^n
\rbrace \simeq \mathbb{Z}_2^n.
\]
Here $\mathbb{Z}_2=\mathbb{Z}/2\mathbb{Z} =  \{0,1\}$ is the additive group and $z^\beta \coloneqq z_1^{\beta_1}\dots z_n^{\beta_n}$ for each $\beta \in \mathbb{N}^n$. Thus a function $f: C_2^n \to \mathbb{C}$ can be expressed as a linear combination of multilinear monomials:
\[
f = \sum_{\alpha \in \mathbb{Z}_2^n} f_\alpha z^{\alpha}.
\]
\subsection{Fourier sum of squares on finite abelian groups }
\label{subsec:FSOS}
This subsection concerns with the theory of FSOS developed in \cite{fawzi2016sparse,sakaue2017exact,yang2022computing}. Let $f$ be a nonnegative function on a finite abelian group $G$. An \emph{FSOS certificate} of $f$ is a finite family $\{h_i\}_{i\in I}$ of complex valued functions on $G$ such that $f =\sum_{i \in I}  |h_i|^2$. The \emph{sparsity} of $\{h_i\}_{i\in I}$ is defined to be $\lvert \bigcup_{i \in I}\operatorname{supp}(h_i) \rvert$. We say that $f$ is \emph{sparse} if $f$ admits an FSOS certificate of low sparsity. The proposition that follows ensures the existence of FSOS certificates.
\begin{proposition}\cite[Proposition~2]{fawzi2016sparse}\label{prop:existence of fsos}
A  nonnegative function $f$ on  a finite abelian group admits an FSOS certificate.
\end{proposition}
A function $f$ on $G$ is nonnegative if and only if there exists a Hermitian positive semidefinite matrix $Q = (Q_{\chi,\chi'})_{\chi,\chi'\in \widehat{G}} \in \mathbb{C}^{|G| \times |G|}$ such that:
\begin{equation}\label{eq:gram}
\sum_{\chi' \in \widehat{G}} Q_{(\chi')^{-1} \chi,\chi}=f_{\chi} ,\quad \forall \chi \in \widehat{G}.
\end{equation}
Any $|G|\times |G|$ complex matrix $Q \succeq 0$ satisfying \eqref{eq:gram} is called a \emph{Gram matrix} of $f$. Gram matrices are of great importance in both the theoretical and computational study of FSOS. In fact, if $Q$ is a Gram matrix of $f$ and $Q = M^* M$ for some matrix
\[
M = (M_{j,\chi})_{1\le j \le r,\chi \in \widehat{G}}\in \mathbb{C}^{r \times |G|},
\]
then we have $  f= \sum_{j=1}^r \left| \sum_{\chi \in \widehat{G}}M_{j,\chi} \chi  \right|^2$. Clearly this construction provides us a one-to-one correspondence between (sparse) Gram matrices and (sparse) FSOS certificates of $f$. Thus the problem of computing a sparse FSOS certificate of $f$ is equivalent to computing a sparse Gram matrix of $f$.
\begin{theorem}\cite[Lemma 3.6]{yang2022computing}\label{thm:yang2022fsos}
Let $f$ be a nonnegative, nonzero function on a finite abelian group $G$. Suppose
\[\sqrt{f} =\sum_{\chi \in \widehat{G}}a_{\chi}\chi\]
 is the Fourier expansion of $\sqrt{f} $.
Then the optimal solution of the convex relaxation of
\[
\min  \limits_{Q:\text{gram matrix of}~f} \|Q\|_0
\]
is the rank-one positive semidefinite matrix $H \in \mathbb{C}^{|\widehat{G}| \times \lvert \widehat{G} \rvert}$ where \[H_{\chi,\chi'}=\overline{a_{\chi}}a_{\chi'}.\] Here again we respectively index rows and columns of $H$ by elements in $\widehat{G}$.
\end{theorem}
By fast Fourier transform (FFT) and the inverse fast Fourier transform (iFFT), $\sqrt{f}$ can be computed in quasi-linear time in $|G|$. However, if group $G=C_2^n$, then  the complexity of computing $\sqrt{f}$ is $\O(n\cdot 2^n)$ which is exponential in $n$. In Section \ref{sec:Select_basis}, we present a method to estimate $\sqrt{f}$ efficiently.
\subsection{\texorpdfstring{integer-valued functions and characteristic functions of weighted $k$-CNF formulae}{}}\label{subsec:characteristic function}
Let $G$ be a finite abelian group. In this paper, an \emph{integer-valued function}\footnote{The assumption on the bound of values of $f$ is used to control the complexity of Algorithm~\ref{alg:selection}. Except for this, all statements in this paper actually hold true for any $f:G\to \mathbb{Z}$.} on $G$ is a function $f: G \to \mathbb{Z}$ such that for all $g\in G$, $|f(g)| = O\left(\operatorname{polylog}(|G|)\right)$.

Typical examples of integer-valued functions are characteristic functions of weighted $k$-CNF formulae. A \emph{weighted $k$-CNF formula} $\phi$ consists of logic clauses
\[
\left\lbrace
\left( \lor_{i \in S_j^+}x_i \right) \lor \left( \lor_{i \in S_j^-}\lnot x_i \right)
\right\rbrace_{j \in J}
\]
and integer weights $\{w_j\}_{j \in J}$. Here $x_i$'s are logic variables with values in $\{\text{True},\text{False}\}$ and $\lor,\lnot$ are logical operations $\text{OR}$ and $\text{NOT}$ respectively. We remark that if $w_j = 1$ for all $j\in J$, then $\phi$ is simply a $k$-CNF \cite{DSW94}:
\[
\sum_{j\in J} \left( \lor_{i \in S_j^+}x_i \right) \lor \left( \lor_{i \in S_j^-}\lnot x_i \right).
\]

We define the \emph{characteristic function} of $\phi$ by
\[f_{\phi}(z_1,\dots, z_n)=\sum_{j \in J} \frac{w_j}{2^l}\left(\prod_{i \in S_j^+}\left(1+z_i\right) \cdot \prod_{i \in S_j^-}\left(1-z_i\right)\right),\]
where $(z_1,\dots, z_n) \in C_2^n$. It is clear that the weighted MAX-SAT problem for $\phi$ is equivalent to the problem of computing the minimum of $f_{\phi}$.

If we define
\[
\tau: \{\text{False},\text{True}\} \to C_2,\quad \tau(\text{False})=1, \tau(\text{True})=-1,
\]
then characteristic functions of $k$-CNF formulae have the following properties.
\begin{proposition}\cite{yang2022computing}\label{prop:arithmetization}
Let $\phi$ be a $k$-CNF formula in $n$ variables with $m$ clauses and let $f_{\phi}$ be the characteristic function of $\phi$. Then
\begin{enumerate}[(i)]
\item For $x_1,x_2,\cdots,x_n \in \{\text{False},\text{True}\}$, $f_{\phi}(\tau(x_1),\tau(x_2),\cdots,\tau(x_n))$ is the number of clauses of $\phi$ falsified by $x_1,x_2,\cdots,x_n$.
\item The degree of $f_{\phi}$ is at most $k$.\label{prop:arithmetization:1}
\item The cardinality of $\operatorname{supp} (f_{\phi})$ is at most $\min\left\lbrace 2^km, \sum_{i=0}^{k}\binom{n}{i}\right\rbrace$. \label{prop:arithmetization:2}
\item $f_{\phi}$ can be computed by $O(2^kkm)$ operations. In particular, if we fix $k$ then $f_{\phi}$ can be computed by $O(m)$ operations. \label{prop:arithmetization:3}
\item  For each $x\in \{ \text{False}, \text{True}\}^n$, $f_{\phi}(\tau(x))$ is the number of clauses of $\phi$ falsified by $x$. \label{prop:arithmetization:4}
\item The image set of $f_{\phi}$ is contained in $\{0,\dots,m\}$. \label{prop:arithmetization:5}
\end{enumerate}
Moreover, the following statements are equivalent:
\begin{enumerate}[(a)]
\item The number of satisfiable clauses of $\phi$ is at most $(m - L)$.\label{lem:verification1-item0}
\item The number of falsified clauses of $\phi$ is at least $L$. \label{lem:verification1-item1}
\item $f_{\phi} - L + \delta \ge 0$ for some function $\delta: C_2^n \to [0,1)$. \label{lem:verification1-item2}
\item $f_{\phi} - L  + \delta \ge 0$ for any function $\delta: C_2^n \to [0,1)$. \label{lem:verification1-item3}
\end{enumerate}
\end{proposition}
According to Proposition~\ref{prop:arithmetization}, solving MAX-SAT for a CNF formula $\phi$ is equivalent to solving Problem~\ref{problem:lower bound} for $f_{\phi}$.
\section{Lower bounds by FSOS}\label{sec:Select_basis}
In this section, we present our solution to Problem~\ref{problem:lower bound} for an integer-valued function $f$ on $G$. Before we proceed, it worths a remark that if a lower bound $L$ of $f$ is given, then a better lower bound of $f$ can be obtained by computing a positive lower bound of $ f - L$ unless $L$ is already the minimum of $f$.

Our solution to Problem~\ref{problem:lower bound} is based on the lemma that follows.
\begin{lemma}\label{lem:sdp formulation}
Let $f:G \mapsto \mathbb{R}$ be a function on a finite abelian group $G$ and let $\alpha_0$ be a real number. Then $\alpha_0$ is a lower bound of $f$ if and only if there exists a finite family $\{h_i\}_{i\in I}$ of complex-valued functions on $G$ such that $f - \alpha_0 = \sum_{i\in I} |h_i|^2$. As a consequence, there is a one to one correspondence between the set of lower bounds of $f$ and the feasible set of
\begin{align}
  &\max_{h_i:G\mapsto \mathbb{C}, i \in I} \alpha, \nonumber \\
 & \text{s.t.}~f-\alpha =\sum_{i \in I} |h_i|^2. \label{lem:sdp formulation:eq1}
\end{align}
\end{lemma}
\subsection{support selection}
We notice that the inclusion of subsets of $\widehat{G}$ induces a filtration on the feasible set of \eqref{lem:sdp formulation:eq1}. Namely, for each $S \subseteq \widehat{G}$, we define
\[
F_S \coloneqq  \Big\lbrace
\{h_i\}_{i\in I}:  f - \alpha = \sum_{i\in I} |h_i|^2~\text{for some}~\alpha, \bigcup_{i\in I} \supp (h_i) \subseteq S
\Big\rbrace.
\]
Then we have
\begin{equation}\label{eq:filtration}
\text{feasible set of \eqref{lem:sdp formulation:eq1}} = F_{\widehat{G}} \coloneqq  \bigcup_{S \subseteq \widehat{G}} F_{S}
\end{equation}
and $F_{S_1} \subseteq F_{S_2}~\text{if}~S_1 \subseteq S_2 \subseteq \widehat{G}$. Therefore for a given $S \subseteq \widehat{G}$, we may consider the following problem:
\begin{equation}\label{OPT:SOS}
  \max_{ \{h_i\}_{i\in I} \in F_S} \alpha.
\end{equation}
By the next lemma, \eqref{OPT:SOS} can be reformulated as an SDP.
\begin{lemma}\label{lem:sdp formulation-1}
For each $S \subseteq \widehat{G}$, \eqref{OPT:SOS} is equivalent to
\begin{eqnarray}\label{SOS:SDP}
  &\max_{Q \in \mathbb{C}^{S\times S}} & \widehat{f}(\chi_0)-\operatorname{trace}(Q), \\
 & \text{s.t.}~&  \sum_{\chi' \in \widehat{G}} Q_{(\chi')^{-1} \chi,\chi}=\widehat{f}(\chi),~\chi \neq \chi_0 \in \widehat{G} \label{SOS:SDP:Cond-1}\\
  && Q \succeq 0  \label{SOS:SDP:Cond-2}.
\end{eqnarray}
Here $\chi_0 \in \widehat{G}$ denotes the identity element in $\widehat{G}$ and $Q \in \mathbb{C}^{S\times S}$ means $Q$ is a matrix of size $|S|$ whose rows and columns are indexed by elements in $S$.
\end{lemma}
\begin{proof}
Conditions \eqref{SOS:SDP:Cond-1} and \eqref{SOS:SDP:Cond-2} ensure that $Q$ is a Gram matrix of $f- \widehat{f}(\chi_0)+\operatorname{trace}(Q)$. Let $Q$ be a feasible solution of \eqref{SOS:SDP} and let $Q=H^*H$ where $H \in \mathbb{C}^{ r \times S}$. We set $h_{i}=\sum_{\chi \in S}H(i,\chi)\chi, i=1,\dots,r$. Then $\{h_i\} \in F_S$ is a  feasible solution of \eqref{OPT:SOS}, with  $\alpha=\widehat{f}(\chi_0)-\operatorname{trace}(Q)$. On the contrary, if $\{h_i\}_{i \in I}$ is a  feasible solution of (\ref{OPT:SOS}), then we may define $H\in \mathbb{C}^{I \times S}$ where $H(i,\chi)=\widehat{h_i}(\chi)$. Thus $Q=H^*H$ is a  feasible solution of (\ref{OPT:SOS}), with  $\widehat{f}(\chi_0)-\operatorname{trace}(Q)=\alpha$.
\end{proof}
According to Lemma~\ref{lem:sdp formulation} and \eqref{eq:filtration}, each optimal solution of \eqref{SOS:SDP} supplies a lower bound of $f$. The quality of such a lower bound depends on the choice of $S$. For instance, if $S \supseteq \bigcup_{i\in I} \supp(h_i)$ for some optimal solution $\{h_i\}_{i\in I}$ of \eqref{lem:sdp formulation:eq1}, then clearly the lower bound given by an optimal solution of \eqref{SOS:SDP} is the minimum of $f$. Moreover, the size of \eqref{SOS:SDP} is completely determined by $|S|$. Thus to efficiently compute a good lower bound of $f$, we need to choose small $S$ which contains as more elements in $\bigcup_{i\in I} \supp(h_i)$ as possible, for some optimal solution $\{h_i\}_{i\in I}$ of \eqref{lem:sdp formulation:eq1}. The rest of this section is devoted to the choice of $S\subseteq \widehat{G}$, under the guidance of this principal.

Let $f$ be a nonnegative function on an abelian group $G$. We may expand $\sqrt{f}$ as:
\[
\sqrt{f}=\sum_{i=1}^{|G|} a_{i}\chi_i,
\]
where $|a_1| \ge |a_2| \ge\cdots \ge \lvert a_{|G|}  \rvert$ and $\widehat{G} = \{\chi_i\}_{i=1}^{|G|}$. By Theorem \ref{thm:yang2022fsos}, $S = \{\chi_i\}_{i=1}^k$ is already a desired choice of $S$ for \eqref{SOS:SDP}. Here $k \le n$ is some given positive integer. However, as we point out at the end of Subsection~\ref{subsec:FSOS}, the complexity of computing the Fourier expansion of $\sqrt{f}$ is quasi-linear in $|G|$, which leads to an exponential complexity in $n$ if $G = C_2^n$. Thus to reduce the complexity, we may approximate the Fourier expansion of $\sqrt{f}$ instead of computing it exactly. To this end, an estimate of the approximation error is necessary, which is the content of the next proposition.
\begin{proposition}\cite[Proposition 4.12]{yang2022short}\label{thm:Interpolation_Error_Bound}
Let $f$ be a nonnegative function on $G$ and let $\varepsilon >0$ be a fixed number. If $p$ is a univariate polynomial such that $\max_{g\in G} \left| p(f(g))-\sqrt{f(g)} \right| \le \varepsilon$, then we have
\[
\max_{\chi \in \widehat{G}} \left| \widehat{p\circ f}(\chi) - \widehat{f}(\chi) \right| \le \varepsilon.
\]
\end{proposition}
According to the above discussions, we may construct $S$ of cardinality $|S| = k$ by Algorithm~\ref{alg:selection}.
\begin{algorithm}[!htbp]
\caption{Selection of support}
\label{alg:selection}
\begin{algorithmic}[1]
\renewcommand{\algorithmicrequire}{\textbf{Input}}
\Require
nonnegative function $f$ on $G$ and parameters $l, m, d, k\in \mathbb{N}$.
\renewcommand{\algorithmicensure}{\textbf{Output}}
\Ensure
$S\subseteq \widehat{G}$ for \eqref{SOS:SDP} with $|S| = k$.
\State approximate $\sqrt{t}$ by a polynomial $p(t)$ of degree at most $d$ at integer points in $[l,m]$. \label{alg:selection:step1}
\State compute the composition $p\circ f = \sum_{i=1}^{|G|} a_i \chi_i$, where $|a_1| \ge \cdots \ge |a_{|G|}|$ and $\deg(\chi_i) \ge \deg(\chi_{i+1})$ if $|a_i| = |a_{i+1}|, 1\le i\le |G|-1$.\label{alg:selection:step2}
\Return $S = \{\chi_1,\dots, \chi_k\}$. \end{algorithmic}
\end{algorithm}
If $f$ is integer-valued and $l \le f \le m$ for some $0 < l \le m$, then one can easily compute the polynomial $p$ in step~\ref{alg:selection:step1} of Algorithm~\ref{alg:selection} by solving a linear programming \cite{yang2022short}. We remind the reader that $f:G \to \mathbb{Z}$ satisfies $f(g) = O(\operatorname{polylog}(|G|))$ (cf. Subsection~\ref{subsec:characteristic function}). Thus step~\ref{alg:selection:step1} of Algorithm~\ref{alg:selection} can be accomplished in $O(\operatorname{polylog}(|G|))$ time.

By Proposition \ref{thm:Interpolation_Error_Bound}, $p\circ f$ is an estimate of $\sqrt{f}$ whose coefficient error is bounded above by $\max_{l \leq i\leq m} \left|p(i)-\sqrt{i}\right|$. This ensures that the set $S$ obtained by Algorithm~\ref{alg:selection} indeed consists of first $k$ terms of $\sqrt{f}$, if $p(t)$ approximates $\sqrt{t}$ sufficiently good at integer points between $l$ and $m$. Furthermore, it is sufficient to take $d$ to be small in practice. In fact, if $f$ is the characteristic function of some CNF formula, then $d=1$ or $2$ already provides us good lower bounds, as we will see in Appendix~\ref{append:experiments}.

Let $f_{\min}$ be the minimum of $f$ on $G$. We notice that on the one hand, for each $a \le f_{\min}$, we may apply Algorithm~\ref{alg:selection} to $f-a$ to obtain an $S_a \subseteq \widehat{G}$. On the other hand,  according to \eqref{lem:sdp formulation:eq1}, \eqref{OPT:SOS} and \eqref{SOS:SDP}, we can recover $f_{\min}$ from the optimal solution of \eqref{SOS:SDP} if and only if $S$ contains $S_{f_{\min}}$. An implication of the next proposition is that for each $a \le f_{\min}$, $S_a$ contains elements in $S_{f_{\min}}$ whose coefficients in $f$ are sufficiently large. In particular, if $f$ is nonnegative, then we can even take $S_0$ to serve as an estimate of $S_{f_{\min}}$.
\begin{proposition}\label{prop:compare S}
Let $f$ be a function on $G$ and let $a > \widetilde{a} \ge 0$ be lower bounds of $f$. Assume that $c_{\chi}$ is the coefficient of $\chi$ in $\sqrt{f-a}$. If $|c_{\chi}|> \sqrt{a - \widetilde{a}}$, then $\chi \in \supp(\sqrt{f-\widetilde{a}})$. Moreover, if $G = C_2^n$ and $f \le m$ for some $m\in \mathbb{R}$, then $|c_{\chi}| > \frac{a - \widetilde{a}}{2} \left( \frac{1}{\sqrt{a - \widetilde{a}}} - \frac{1}{\sqrt{m - a} + \sqrt{m - \widetilde{a}}} \right)$ implies $\chi \in \supp(\sqrt{f - \widetilde{a}})$.
\end{proposition}
\begin{proof}
We recall that $c_{\chi} \coloneqq \frac{1}{|G|} \sum_{g \in G} \overline{\chi(g)} \sqrt{f(g) - a}$. Similarly, we may write $\sqrt{f - \widetilde{a}} = \sum_{\chi\in \widehat{G}} \widetilde{c_{\chi}} \chi$. For each $\chi \in \widehat{G}$, we notice that
\begin{align*}
|c_{\chi} - \widetilde{c_{\chi}}| &= \frac{1}{|G|} \Big\lvert \sum_{g\in G} \overline{\chi(g)} \left(
\sqrt{f(g) -a} - \sqrt{f(g) - \widetilde{a}} \right) \Big\rvert \\
&= \frac{a - \widetilde{a}}{|G|} \Big\lvert \sum_{g\in G} \frac{ \overline{\chi(g)}}{ \sqrt{f(g) -a} + \sqrt{f(g) - \widetilde{a}} } \Big\rvert \\
&\le \sqrt{a - \widetilde{a}}.
\end{align*}
Therefore, if $|c_{\chi}| > \sqrt{a - \widetilde{a}}$ then $\widetilde{c_{\chi}} \ne 0$.

If $G = C_2^n$, then $\chi\in \widehat{G} \simeq \mathbb{Z}_2^n$ is a multilinear monomial. Thus $\chi(g) = \pm 1$ for any $g\in C_2^n = \{-1,1\}^n$. Moreover, it is straightforward to verify that we can construct a bijective map $\psi: C_2^n \to C_2^n$ such that $\chi(g) = - \chi(\psi(g))$ and $\psi^2 = \Id$. Thus we have
\begin{align*}
|c_{\chi} - \widetilde{c_{\chi}}| &= \frac{a - \widetilde{a}}{|G|} \Big\lvert \sum_{g\in G} \frac{ \overline{\chi(g)}}{ \sqrt{f(g) -a} + \sqrt{f(g) - \widetilde{a}} } \Big\rvert \\
&\le \frac{a - \widetilde{a}}{|G|} \sum_{g\in G, \chi(g) = 1} \left( \frac{1}{\sqrt{a - \widetilde{a}}} - \frac{1}{\sqrt{m - a} + \sqrt{m - \widetilde{a}}} \right) \\
&= \frac{a - \widetilde{a}}{2} \left( \frac{1}{\sqrt{a - \widetilde{a}}} - \frac{1}{\sqrt{m - a} + \sqrt{m - \widetilde{a}}} \right).
\end{align*}
\end{proof}
\subsection{FSOS with error\label{sec:SOS_with_error}}
We discuss in this subsection a remarkable feature of our solution (cf. Section~\ref{sec:Select_basis}) to Problem~\ref{problem:lower bound}. That is, a solution to the SDP problem~\eqref{SOS:SDP} which violates conditions   \eqref{SOS:SDP:Cond-1} and \eqref{SOS:SDP:Cond-2} can still provide us a tight lower bound
of $f$. This is the content of the next Theorem.
\begin{theorem}[FSOS with error]\label{thm:SOS_with_error}
Let $G$ be a finite abelian group and let $S$ a subset of $\widehat{G}$. Given a function $f:G \mapsto \mathbb{R}$ and a Hermitian matrix $Q\in \mathbb{C}^{S \times S}$, we have
\[
\min_{g\in G} f(g)\geq -\|\widehat{e}\|_1+\lambda |S|,
\]
where $\lambda$ is the minimal eigenvalue of $Q$, $e=f-v_S^* Q v_S$ and $v_S =(\chi)_{\chi \in S}$ is the vector consisting of all characters in $S$.
\end{theorem}
\begin{proof}
For any $g\in G$, we have
\[f(g)-e(g)=v_S(g)^*Qv(g)\geq \lambda v_S(g)^*v(g).\]
We observe that $|\chi(g)|=1$ for each $g\in G$ and $\chi \in \widehat{G}$. Thus $v_S(g)^*v_S(g) \leq |S|$ and $|e(g)| \leq \|\widehat{e}\|_1$ for any $g\in G$. This implies
\[
f(g)\geq -\|\widehat{e}\|_1+\lambda |S|,\quad g\in G.
\]
\end{proof}
According to Theorem~\ref{thm:SOS_with_error}, $-\|\widehat{e}\|_1+\lambda  |S|$ is a lower bound of $f$  even if $Q\nsucceq 0$ or  $e \neq 0$. The example that follows indicates that our method can give a tight lower bound of $f$ even if it does not admit an FSOS supported on $S$.
\begin{example}\label{ex4.4}
Let $f:C_2^3\to \mathbb{R}$ be the function defined by
\[
f(z_1,z_2,z_3)=4+z_1+z_2+z_3+z_1z_2z_3.
\]
We can check that the SDP~\eqref{SOS:SDP} has no feasible solution for $S=\{1,z_1,z_2,z_3\}$, i.e., $f$ has no FSOS supported on $S$. However, if we let $e(z_1,z_2,z_3) \coloneqq z_1z_2z_3$, then
\[
f-e=1+ \frac{1}{2} \sum_{i=1}^3 (1+z_i)^2,
\]
from which we obtain $f\geq 1 -\|\widehat{e}\|_1 =0$.
\end{example}
Below we record an algebraic version of Theorem~\ref{thm:SOS_with_error}, which is the form we need in Section~\ref{sec:Numerical}.
\begin{corollary}\label{coro:SOS_with_error}
Let $f,S,Q,e,\lambda$ be defined as above. Then $Q-\lambda \operatorname{Id} \succeq 0$ is  a Gram matrix of the function $f-e-\lambda |S|$.
\end{corollary}
We observe that  $f$ may be represented by the polynomial function $F_{Q - \lambda \Id} + F_{\lambda \Id}$, where $F_M(z) = z^\tp M z$ is the non-negative polynomial determined by a symmetric matrix $M$. Since $Q - \lambda \Id \succeq 0$, we obtain that the minimum of $F_{\lambda \Id}$ supplies a lower bound of $f$. Unfortunately,
when $z\in \mathbb{R}^n$,
the minimum of $F_{\lambda \Id}$ is $-\infty$ if $\lambda < 0$. This phenomenon distinguishes our method from the traditional method of polynomial optimization.

For instance, we consider $f = 2z_1z_2$ on $C_2^2 $ so that $Q$ is
\[\kbordermatrix{\mbox{}&z_1&z_2\\
z_1&0& 1\\
z_2&1 & 0}\]
and the minimal eigenvalue of $Q$ is $-1$. According to Corollary~\ref{coro:SOS_with_error}, we conclude that $Q + \Id\succeq 0$ is the Gram matrix of $f + 2$ and $f\ge -2$ on $C_2^2$. As a comparison, $f$ is represented by the polynomial
\[
F_{Q + \Id} + F_{-\Id} = (z_1+z_2)^2-(z_1^2+z_2^2)
\]
and the minimum of $F_{-\Id} = -(z_1^2+z_2^2)$ is $-\infty$, which fails to provide a non-trivial lower bound of $f$.

We conclude this section by briefly summarizing the main advantages of FSOS with error. Interested readers are referred to Appendix~\ref{append:experiments} for numerical examples.
\begin{enumerate}
  \item \emph{Early termination}: existing methods \cite{VVH2008,wang2019low} need to wait the algorithm converges to  a solution satisfying conditions  \eqref{SOS:SDP:Cond-1} and \eqref{SOS:SDP:Cond-2}. According to Theorem \ref{thm:SOS_with_error}, our method can find a lower  bound even these conditions are not satisfied. This feature enables us to set different time limits on solving SDP problems. Clearly the longer time limit we set, the better lower bound we obtain.

  \item \emph{Adaptivity to more SDP solvers}: since the condition \eqref{SOS:SDP:Cond-2} is not required to be satisfied exactly, we can use the more efficient SDP solver SDPNALplus \cite{sun2020sdpnal+} to solve the problem in Lemma  \ref{lem:sdp formulation-1}. As a consequence, MAX-kSAT problems of much larger sizes from the benchmark set\footnote{\href{http://www.maxsat.udl.cat/09/index.php} {http://www.maxsat.udl.cat/09/index.php}}
     \footnote{ \href{http://www.maxsat.udl.cat/16/index.html}{http://www.maxsat.udl.cat/16/index.html}} can be solved successfully, with which SOS based algorithms proposed in \cite{VVH2008} fail to deal.
%
%

  \item \emph{Size reduction}: given a monomial basis $S$, algorithms in \cite{VVH2008} can not get a lower bound if $\operatorname{supp}(f) \not\subseteq  \{\chi \chi':\chi,\chi' \in S\}$, since $f$ has no FSOS supported on $S$. However, according to Theorem \ref{thm:SOS_with_error}, a feasible solution of \eqref{SOS:SDP} that fails to satisfy \eqref{SOS:SDP:Cond-1} can still provide us a lower bound. This leads to a reduction on sizes of our SDP problems.
\end{enumerate}
\section{Computation of lower bounds}\label{sec:Numerical}
This section is concerned with turning discussions in Section~\ref{sec:Select_basis} into an algorithm for computing a lower bound of a function $f$ on a finite abelian group $G$. To begin with, we recall that every function $h$ on $G$ admits the Fourier expansion:
\[
h = \sum_{\chi \in \widehat{G}} \widehat{h}(\chi) \chi.
\]
For simplicity, we denote $h_0 \coloneqq \widehat{h}(1)$, where $1$ is the identity element in $\widehat{G}$.
\subsection{our algorithm}
Let $S$ be a subset of $\widehat{G}$ and let $Q\in \mathbb{C}^{S \times S}$ be a Hermitian matrix. We define $e \coloneqq f-f_0-v_S^* Q v_S$ where $v_S$ is the $|S|$-dimensional column vector consisting of characters in $S$. Applying Theorem~\ref{thm:SOS_with_error} to $f-f_0$, we obtain
\[
f(g) \geq f_0 +e_0-\|\widehat{e-e_0}\|_1 + \lambda_{\min}(Q) |S|,\quad g\in G
\]
where $\lambda_{\min}(Q)$ is the smallest eigenvalue of $Q$. We also observe that
\[
e_0 =-\sum_{\chi=\chi'}Q(\chi,\chi')=-\operatorname{trace}(Q).
\]
Therefore, we can find a lower bound of $f$ by solving the following unconstrained convex optimization problem:
\begin{equation}\label{Eq:opti:function}
  \min _{Q \in \mathbb{C}^{S\times S},~Q=Q^*} F (Q),
\end{equation}
where
 \[ F(Q)=\operatorname{trace}(Q)+\|E(Q)\|_1-\lambda_{\text{min}}(Q)  |S|\]
 and
$E:\mathbb{C}^{S\times S} \mapsto \mathbb{C}^{|G|}$ is the affine map which maps $Q$ to the (sparse) vector consisting of Fourier coefficients of $e-e_0$. We notice that the subgradient of $F$ is
\[
\partial F=\operatorname{Id}+ \operatorname{sign}(E(Q)) \partial E -|S| uu^*,
\]
where $\partial E$ is  the gradient of $E$,   $\operatorname{sign}(x)$ is the sign function and $u$ is the unit eigenvector of $Q$ corresponding to $\lambda_{\min}(Q)$. Thus we obtain Algorithm~\ref{alg:Fast_Lower_Bounds} for a lower bound of $f$.
\begin{algorithm}[!htbp]
\caption{Lower Bounds of Functions on Finite Abelian Groups.
}
\label{alg:Fast_Lower_Bounds}
\begin{algorithmic}[1]
\renewcommand{\algorithmicrequire}{\textbf{Input}}
\Require
 a function $f$ on $G$, positive integers $l$, $m$, $d$ and $k$.
\renewcommand{\algorithmicensure}{\textbf{Output}}
\Ensure
a lower bound of $f$.
\State Select $S$ with $l$, $m$, $d$ and $k$. \Comment{by Algorithm \ref{alg:selection}} \label{alg:Fast_Lower_Bounds:step1}
\State  Solve \eqref{SOS:SDP} for $Q_0$. \Comment{by SDPNAL+}  \label{alg:Fast_Lower_Bounds:step2}
\State Solve \eqref{Eq:opti:function} for $Q$. \Comment{by gradient descent with initial point $Q_0$} \label{alg:Fast_Lower_Bounds:step3} \\
\Return $f_0-F(Q)$.
\end{algorithmic}
\end{algorithm}
\subsection{rounding}\label{sec:rounding}
Rounding is an important step in {SDP-based} algorithms for combinatorial optimization problems, especially when both the optimum value and optimum point are concerned. There exist several rounding techniques in the literature. Examples include rounding by random hyperplanes \cite{goemans1995improved} together with its improved version \cite{377033}, the skewed rounding procedure \cite{LLZ02} and the randomized rounding technique \cite{VVH2008}. Among all these rounding strategies, the one proposed in \cite{VVH2008} can be easily adapted to our situation.

We present two rounding methods, one is  based on the null space of the Gram matrix, and the other is based on the lower rank approximation of the  moment matrix.

\begin{enumerate}

\item Let $Q  \in \mathbb{R}^{S \times S}$ be a solution to
\eqref{SOS:SDP}. For a given $f: G\to \mathbb{R}$ and $S\subseteq \widehat{G}$ containing all characters of degree at most one, we assume
\[
f-\alpha=v_S^* Q  v_S,
\]
where $\alpha$ is the minimum value of $f$, $v_S$ is the $|S|$-dimensional column vector consisting of characters in $S$ and $Q \succeq 0$. 


If the null space of $Q$ is one-dimensional, then by normalizing the first element of any null vector of $Q$ to be one, we obtain the desired solution. Once we obtain a desired null vector $v\in \mathbb{R}^{|S|}$, we can  recover $\tilde{g}\in G$ by rounding elements of $v$.

If the null space of $Q$ has dimension $d > 1$, then the normalization is no longer sufficient since elements of a null vector of $Q$ might be algebraically inconsistent in this case. We can extract the solution  based on  the Stickelberger theorem,  which is often used to solve polynomial systems in the literature \cite{Stetter:1995,Corless1995,Corless1997,ReidZhi08,Henrion2005}.

\item The moment matrix  $H \succeq 0$   is  the solution of the dual problem of (\ref{SOS:SDP}).
  Let $H=\sum_{i=1}^{|S|} \mu_i u_iu_i^*$ be the eigenvalue decomposition of $H$, with $\mu_1 \geq \mu_2 \geq \mu_3,...\geq \mu_{|S|}$. Then $\widetilde{H}=\mu_1 u_1 u_1^* \in \mathbb{R}^{S \times S}$ is  the rank-one approximation  of $H$. We can recover the solution $\widetilde{x}$ from $\widetilde{H}$ by setting  $\widetilde{x}_i=\widetilde{H}_{1,x_i}$.

\end{enumerate}

 In practice, the matrix $Q$ obtained in Algorithm~\ref{alg:Fast_Lower_Bounds} may not be positive semidefinite. We need to update it by  $Q + \lambda_{\min}(Q) \Id$. It is a little surprise for us to notice that
 the numerical corank of   $Q + \lambda_{\min}(Q) \Id$ is   $1$ very often,  which makes it possible to recover the optimal solution efficiently from its null vector.

To conclude this section, we illustrate the above rounding procedure by examples.
\begin{example}[one-dimensional null space]
In Example~\ref{ex4.4} we see that
\[
f(z_1,z_2,z_3)=4+z_1+z_2+z_3+z_1z_2z_3
\]
is a non-negative function on $C_2^3$. For
\[
S=\{1,z_1,z_2,z_3,z_1z_2z_3\},
\]
we may obtain a Hermitian (but not positive semidefinite) matrix by SDPNAL+:
\[Q=\kbordermatrix{\mbox{}&1&z_1&z_2&z_3&z_1z_2z_3\\
1& 1.9546 & 0.5000 & 0.5000 & 0.5000 & 0.5000\\
z_1& 0.5000 & 0.4536 & 0.0000 & 0.0000 & 0.0000 \\
z_2& 0.5000 & 0.0000 & 0.4536 & 0.0000 & 0.0000 \\
z_3& 0.5000 &   0.0000  &  0.0000   & 0.4536 &   0.0000\\
z_1z_2z_3&0.5000 & 0.0000 & 0.0000 & 0.0000 & 0.4536},
\]
whose eigenvalues are
\[-0.0462,0.4536, 0.4536, 0.4536,  2.4544.\]
 The normalized eigenvector corresponding to $-0.0462$ is
\[
v = \kbordermatrix{\mbox{}&1&z_1&z_2&z_3&z_1z_2z_3\\
\mbox{}&1 &-1.000447&-1.000447& -1.000447&-1.000447}.
\]
We recover $z_1=z_2=z_3=-1$ by rounding the elements of $v$, which is the optimal solution of  $f=0$.
\end{example}

\begin{example}[higher dimensional null space]
We consider the function $f(z_1,z_2) = (1+z_1+z_2+z_1z_2)^2$ on $C_2^2$ and $S = \{1,z_1,z_2,z_1z_2\}$. We notice that $f$ has a Gram matrix
\[
Q = \begin{bmatrix}
1 & 1 & 1 & 1 \\
1 & 1 & 1 & 1 \\
1 & 1 & 1 & 1 \\
1 & 1 & 1 & 1 \\
\end{bmatrix},
\]
The null space of $Q$ is spanned by
\[
\begin{bmatrix}
1 & 0 & 0 & -1
\end{bmatrix}^\tp,\quad \begin{bmatrix}
1 & 0 & -1 & 0
\end{bmatrix}^\tp, \quad \begin{bmatrix}
1 & -1 & 0 & 0
\end{bmatrix}^\tp.
\]
Using these null vectors, we form the multiplication matrix\footnote{Any random linear combination of $z_1$ and $z_2$ works as well.} of $z_1 - z_2/2$:
\[
\begin{bmatrix}
0 & 1 & -1/2 \\
3/2 & 1/2 & 1/2 \\
-3/2 & -1 & -1
\end{bmatrix}
\]
whose eigenvectors are
\[
\begin{bmatrix}
-1 & -1 & 1
\end{bmatrix}^\tp, \quad \begin{bmatrix}
-1 & 1 & 1
\end{bmatrix}^\tp, \quad \begin{bmatrix}
1 & -1 & 1
\end{bmatrix}^\tp.
\]
By normalizing the first element of these eigenvectors to be $1$, we recover  three  solutions $(1,-1), (-1,-1)$ and $(-1,1)$ of $f (z_1,z_2) = 0$.
\end{example}
%

\appendix
\section{Numerical experiments}\label{append:experiments}
In this section, we perform three numerical experiments to test Algorithm~\ref{alg:Fast_Lower_Bounds} and the rounding {techniques} discussed in Section~\ref{sec:rounding}. We implement our algorithm in Matlab (2016b) with SDPNAL+ \cite{sun2020sdpnal+}. The code is available on \href{http://github.com/jty-AMSS/Fast-Lower-Bound-On-FAG}{github.com/jty-AMSS/Fast-Lower-Bound-On-FAG}.
 All experiments are performed on a desktop computer with Intel Core i9-10900X@3.70GHz CPU and 128GB RAM memory.
\subsection{lower bounds of random functions}\label{subsec:random functions}
The goal of the first experiment is to exhibit the correctness and efficiency of Algorithm~\ref{alg:Fast_Lower_Bounds}. We randomly generate non-negative integer-valued functions on two finite abelian groups and compute their lower bounds by Algorithm~\ref{alg:Fast_Lower_Bounds}, TSSOS\cite{WML21} and CS-TSSOS \cite{WMLM22} respectively. Without loss of generality, we only consider functions whose minimum values are zero.

For comparison purposes, we record the running time and the result for each algorithm in Tables~\ref{Tab:random-example-1} and \ref{Tab:random-example-2}. This experiment indicates that Algorithm~\ref{alg:Fast_Lower_Bounds} is more efficient than the general-purpose methods TSSOS and CS-TSSOS on random examples.
\subsubsection{experiment on $C_2^{25}$}
We generate ten polynomials of degree three and sparsity at most $500$ on $C_2^{25}$ by randomly picking its non-constant coefficients from the set $\{m\in \mathbb{Z}: -5 \le m \le 5\}$. The constant terms of these polynomials  are chosen so that their minimum values are zero.

For each of these ten polynomial $f$, we perform Algorithm~\ref{alg:Fast_Lower_Bounds} with the following parameters:
\begin{align*}
d&=2, \quad k=3\left|\operatorname{supp}(f)\right|,\quad  l = 0, \\
m &= \text{sum of absolute value of coefficients of $f$}.
\end{align*}
We apply TSSOS and CS-TSSOS with Mosek \cite{mosek} of relaxation order $2$ to compute lower bounds of these functions. Results are shown in Table \ref{Tab:random-example-1}, where``sp" means the sparsity and ``bound" means the lower bound obtained by the corresponding algorithm.
\begin{table}[htbp]
\begin{center}
\begin{tabular}{|c|c|c|c|c|c|c|c|}
  \hline
 \multirow{2}{*}{No}&  \multirow{2}{*}{sp}& \multicolumn{2}{|c|}{Algorithm~\ref{alg:Fast_Lower_Bounds}} & \multicolumn{2}{|c|}{TSSOS} &\multicolumn{2}{|c|}{CS-TSSOS}\\ \cline{3-8}
   && bound & time & bound & time& bound & time\\ \hline
1&451&-6.6e-01&903&  \underline{-2.75} &\textbf{407}&9.0e-09&821 \\ \hline
2&440&-1.1e-01&\textbf{587}& \underline{-3.95}&1208& \underline{-3.44} &1302 \\ \hline
3&453&-4.9e-02&\textbf{519}&6.3e-08&1104&4.5e-09&1634 \\ \hline
4&453&-1.6e-06&1464& \underline{-2.5} &\textbf{1091}& \underline{-1.92} &1510 \\ \hline
5&451&-4.4e-01&804& \underline{-4.06} &\textbf{636}&2.2e-06&1535 \\ \hline
6&457&-6.7e-01&\textbf{581}&9.0e-11&1081&2.3e-10&2031 \\ \hline
7&452&-2.2e-02&\textbf{636}&4.0e-10&847&7.0e-09&1895 \\ \hline
8&455&-6.5e-03&\textbf{775}&4.8e-09&797&7.0e-11&1126 \\ \hline
9&443&-1.3e-01&\textbf{554}&4.9e-10&716&2.4e-11&1284 \\ \hline
10&454&-1.9e-02&\textbf{561}&2.0e-08&807&3.0e-09&1352 \\ \hline
\end{tabular}
\end{center}
\caption{Random examples on $C_2^{25}$} \label{Tab:random-example-1}
\end{table}

We notice that in Table~\ref{Tab:random-example-1}, Algorithm~\ref{alg:Fast_Lower_Bounds} successfully recovers the minimum value zero for all these ten functions since our functions are integer-valued. As a comparison, TSSOS and CS-TSSOS fail to recover the minimum value on four and two instances, respectively. Moreover, Algorithm~\ref{alg:Fast_Lower_Bounds} is  faster than TSSOS (resp. CS-TSSOS) on seven (resp. nine) out of ten functions.
\subsubsection{experiment on $C_3^{15}$}
We generate ten functions on $C_3^{15}$ by the following procedure:
\begin{enumerate}
  \item Set  $f=0$; \label{Generate_random_function:step-1}
  \item Randomly generate an integer-valued function $h$ on $C_3^{2}$, such that $0\leq h \leq 10$; \label{Generate_random_function:step-2}
  \item Randomly pick a  projection map $\tau: C_3^{15} \to C_3^{2} $; \label{Generate_random_function:step-3}
  \item Update $f\leftarrow f+h\circ \tau$; \label{Generate_random_function:step-4}
  \item Repeat steps \eqref{Generate_random_function:step-2}-\eqref{Generate_random_function:step-4} until the sparsity of $f$ is greater than $190$. 
 \item Update $f\leftarrow f- \min_{g \in C_3^{15}} f(g) $.
\end{enumerate}
Functions obtained by the above procedure are integer-valued and their minimum values are zero. Although Algorithm~\ref{alg:Fast_Lower_Bounds} can handle functions on $C_3^{15}$ of sparsity up to $350$, we set the sparsity limit to be $190$ because CS-TSSOS can only deal with functions of sparsity around $200$.

We compute their lower bounds by Algorithm~\ref{alg:Fast_Lower_Bounds} with
\begin{align*}
d&=2,\quad k=3\left|\operatorname{supp}(f)\right|,\quad l=0 \\
m &= \text{sum of maximum value of $h$ in step \eqref{Generate_random_function:step-2}}.
\end{align*}
We also apply CS-TSSOS with Mosek of relaxation order\footnote{CS-TSSOS fails to complete the computation if we use smaller relaxation order.} $4$ to these functions.

We do not test TSSOS in this experiment since it is only applicable to functions with real variables and {real coefficients}, and $C_3 = \{1,\omega,\omega^2\}$ where $\omega = \exp(2\pi i/3) \in \mathbb{C}$. \if CS-TSSOS can only compute the first step of the  hierarchy, and returned "out of memory" when we try to compute higher steps. \fi
Results are set out in Table \ref{Tab:random-example-2}, where ``sp" denotes the sparsity of each function and ``bound" means the lower bound computed by the corresponding algorithm.
\begin{table}[htbp]
\begin{tabular}{|c|c|c|c|c|c|c|}
  \hline
 \multirow{2}{*}{No}&  \multirow{2}{*}{sp}& \multicolumn{2}{|c|}{Algorithm~\ref{alg:Fast_Lower_Bounds}} & \multicolumn{2}{|c|}{CS-TSSOS} \\ \cline{3-6}
   & &bound & time &bound & time \\ \hline
1&191&-1.20e-01& \textbf{237.0}&-1.98e-05&3148.1 \\ \hline
2&191&-3.42e-12&\textbf{224.8}&2.38e-08&1923.3 \\ \hline
3&191&-3.53e-03&\textbf{223.6}&2.76e-09&444.5 \\ \hline
4&191&-1.74e-01&\textbf{887.1}&-8.55e-06&6623.8 \\ \hline
5&191&-3.31e-12&\textbf{876.6}&1.17e-07&1745.5 \\ \hline
6&191&-2.40e-12& \textbf{221.7}&5.34e-07&752.4 \\ \hline
7&191&-3.82e-12&\textbf{225.3}&3.92e-08&362.9 \\ \hline
8&191&-6.92e-02&\textbf{402.5}&8.33e-09&2144.8 \\ \hline
9&191&-1.53e-01&\textbf{227.0}&4.21e-07&1695.6 \\ \hline
10&191&-5.71e-12&\textbf{404.3}&4.61e-10&1105.3 \\ \hline
\end{tabular}
\caption{Random examples on $C_3^{15}$} \label{Tab:random-example-2}
\end{table}

As illustrated in Table~\ref{Tab:random-example-2}, both Algorithm~\ref{alg:Fast_Lower_Bounds} and CS-TSSOS successfully recover the minimum value for all instances, but Algorithm~\ref{alg:Fast_Lower_Bounds} is much faster than CS-TSSOS.

We notice that all the ten functions in Table~\ref{Tab:random-example-2} have the same sparsity $191$. After a little thought, one can realize that this is a result of our procedure to generate these functions. Indeed, the function $h\circ \tau$ in the procedure contains $8$ monomials $1,z_i,z_j,z_i^2, z_j^2, z_iz_j, z_i^2z_j, z_iz_j^2$ for some $1\le i < j \le 15$. Once $f$ involves all $15$ variables in some iteration, it must contain all $31$ monomials $z_j^k,1\le j \le 15, 0\le k \le 2$. Thus each update $f + h\circ \tau$ only increases the sparsity by $0$ or $4$ and the sparsity of the resulting function is of the form $31 + 4k, k\in \mathbb{N}$. Since we terminate the procedure if the sparsity exceeds $190$, the sparsity is expected to be $191 = 31 + 4\times 40$.
\subsection{upper bounds of MAX-SAT problems}\label{subsec:MAXSAT}
As an application, we test Algorithm~\ref{alg:Fast_Lower_Bounds} on MAX-SAT problems. We first recall from Proposition~\ref{prop:arithmetization}-\eqref{prop:arithmetization:4} that for a given CNF formula $\phi$ in $n$ variables with $m$ clauses, the maximum number of simultaneously satisfiable clauses in $\phi$ is equal to $m - \min_{g\in C_2^n} f_{\phi}(g)$. Thus solving the MAX-SAT probelm for $\phi$ is equivalent to computing the minimum value of its characteristic function $f_{\phi}$ on $C_2^n$.

Our testing MAX-SAT problems are drawn from the benchmark problem set in 2016 and 2009 MAX-SAT competitions\footnote{\href{http://www.maxsat.udl.cat/09/index.php}{http://www.maxsat.udl.cat/09/index.php}}
     \footnote{\href{http://www.maxsat.udl.cat/16/index.html}{http://www.maxsat.udl.cat/16/index.html}}. When applicable, we also apply TSSOS and CS-TSSOS to solve these problems. Results are summarized in Tables~\ref{Tab:Max-2SAT} and \ref{Tab:Max-3SAT}. From the former, we may conclude that Algorithm~\ref{alg:Fast_Lower_Bounds} is more effective than TSSOS and CS-TSSOS on MAX-SAT problems, while from the latter, we may see the flexibility of Algorithm~\ref{alg:Fast_Lower_Bounds}.
\subsubsection{experiment on unweighted MAX-2SAT problems}\label{subsubsec:MAX-2SAT}
In this experiment, we consider {randomly  generated} unweighted MAX-2SAT problems in 2016 MAX-SAT competition. Such a problem has $120$ variables, in which the number of clauses ranges from $1200$ to {$2600$}. We apply Algorithm~\ref{alg:Fast_Lower_Bounds} to compute lower bounds of the corresponding characteristic functions on {$C_2^{120}$}, with the following parameters:
\begin{align*}
d&=1, \quad k=\left|\operatorname{supp}(f)\right|,\quad l=0, \\
m&= \text{number of clauses}.
\end{align*}

We apply TSSOS and CS-TSSOS with Mosek of relaxation order $1$ to these functions. It is worthwhile to point out that any higher order relaxations would result in a memory leak. This is because the size of a Gram matrix in a higher order relaxations increases exponentially.\footnote{There are $\sum_{j=0}^k \binom{120}{k}$ monomials involved in the $k$-th relaxation.} For example, the size of a Gram matrix in the first relaxation is $121 \times 121$ while it is $7261 \times 7261$ in the second relaxation.

Numerical results are reported in Table~\ref{Tab:Max-2SAT}, in which ``clause" denotes the number of clauses in each CNF formula, ``min" is the minimum of the characteristic function and ``bound" means the lower bound of the characteristic function obtained by each method.
\begin{table}[htbp]
\begin{center}
\begin{tabular}{|c|c|c|c|c|c|c|c|c|}
  \hline
 \multirow{2}{*}{\small No}&  \multirow{2}{*}{\small  clause}&  \multirow{2}{*}{\small min}& \multicolumn{2}{|c|}{Algorithm~\ref{alg:Fast_Lower_Bounds}} & \multicolumn{2}{|c|}{TSSOS} &\multicolumn{2}{|c|}{CS-TSSOS}\\ \cline{4-9}
 && &\small bound & \small  time &\small  bound & \small  time&\small  bound &\small   time\\ \hline
 1&1200&161& \textbf{159.5} &370&146.7&45&146.7&52 \\ \hline
2&1200& 159& \textbf{156.7} &327&143.1&49&143.1&55 \\ \hline
3&1200&160& \textbf{159.0}&362&146.8&46&146.8&64 \\ \hline
4&1300&180& \textbf{177.5}&450&162.4&52&162.4&73 \\ \hline
5&1300&172& \textbf{170.6} &417& 156.2&47&156.2&65 \\ \hline
6&1300&173& \textbf{171.6}&432&158.8&44&158.8&58 \\ \hline
7&1400&197& \textbf{194.8} &506&179.8&46&179.8&75 \\ \hline
8&1400&191& \textbf{189.3} &499&174.3&51&174.3&87 \\ \hline
9&1400&189& \textbf{187.2} &504&172.1&58&172.1&78 \\ \hline
10&1500&211& \textbf{209.9} &589&194.5&53&194.5&93 \\ \hline
11&1500&213& \textbf{210.1} &573&194.4&50&194.4&88 \\ \hline
12&1500&207& \textbf{205.7} &531&191.3&50&191.3&87 \\ \hline
13&1600&233& \textbf{231.2} &668&215.6&50&215.6&90 \\ \hline
14&1600&239& \textbf{235.0} &668&218.7&48&218.7&85 \\ \hline
15&1600&233& \textbf{230.5} &617&215.2&52&215.2&68 \\ \hline
16&1700&257& \textbf{255.1} &745&238.3&48&238.3&91 \\ \hline
17&1700&248& \textbf{245.7} &749&229.2&53&229.2&105 \\ \hline
18&1700&239& \textbf{238.9}&738&225.3&50&225.3&86 \\ \hline
19&1800&291& \textbf{285.8} &788&268.4&56&268.4&87 \\ \hline
20&1800&262& \textbf{261.3} &896&244.2&53&244.2&99 \\ \hline
21&1800&279& \textbf{277.5} &836& 259.8&50&259.8&99 \\ \hline
22&1900&293& \textbf{292.1} &1002&275.3&67&275.3&119 \\ \hline
23&1900&296& \textbf{294.3} &987&275.8&53&275.8&132 \\ \hline
24&1900&294& \textbf{291.8} &994&273.0&53&273.0&113 \\ \hline
25&2000&307& \textbf{306.4} &1103&288.8&57&288.8&149 \\ \hline
26&2000&321& \textbf{318.1}&1131&299.8&55&299.8&140 \\ \hline
27&2000&307& \textbf{306.0}&1240&288.3&58&288.3&139 \\ \hline
28&2100&336& \textbf{335.5} &1202&317.7&68&317.7&141 \\ \hline
29&2100&336& \textbf{332.2} &1215&313.4&57&313.4&141 \\ \hline
30&2100&332& \textbf{330.3} &1224&311.3&57&311.3&104 \\ \hline
31&2200&358& \textbf{355.6} &1467&334.6&54&334.6&163 \\ \hline
32&2200&371& \textbf{365.8} &1351&345.2&52&345.2&114 \\ \hline
33&2200&359& \textbf{358.1} &1381&338.3&59&338.3&130 \\ \hline
34&2300&380& \textbf{377.6}&1412&356.8&58&356.8&127 \\ \hline
35&2300&383& \textbf{381.5} &1535&359.2&61&359.2&136 \\ \hline
36&2300&365&\textbf{364.7}&1510&345.4&53&345.4&146 \\ \hline
37&2400&389& \textbf{387.9} &1686&368.7&56&368.7&151 \\ \hline
38&2400&402& \textbf{400.0} &1757&378.1&56&378.1&125 \\ \hline
39&2400&380&\textbf{380.0}&1951&363.0&69&363.0&153 \\ \hline
40&2500&418&\textbf{416.4}&1963&395.5&60&395.5&133\\ \hline
41&2500&435&\textbf{432.4}&1876&411.5&55&411.5&143\\ \hline
42&2500&425&\textbf{424.2}&1934&402.4&68&402.4&186\\ \hline
43&2600&439&\textbf{436.7}&2004&414.0&57&414.0&148\\ \hline
44&2600&458&\textbf{455.8}&2010&434.2&57&434.2&221\\ \hline
45&2600&440&\textbf{439.1}&1983&418.7&59&418.7&154\\ \hline
\end{tabular}
\end{center}
\caption{Unweighted MAX-2SAT problems} \label{Tab:Max-2SAT}
\end{table}

From Table~\ref{Tab:Max-2SAT}, we see that lower bounds obtained by Algorithm~\ref{alg:Fast_Lower_Bounds} are very close to minimum values of characteristic functions, which are  better than those obtained by TSSOS and CS-TSSOS. In theory, one can expect an improvement in the quality of lower bounds by increasing the order of relaxations in TSSOS and CS-TSSOS. Unfortunately, as we point out, higher order relaxations are not possible due to the memory leak caused by huge sizes of MAX-2SAT problems we considered.
\subsubsection{experiment on weighted MAX-3SAT problems}
We test Algorithm~\ref{alg:Fast_Lower_Bounds} on weighted MAX-3SAT benchmark problems in 2009 MAX-SAT competition. These 3-CNF formulae are in $70$ variables with $400$ clauses. We remark that neither TSSOS nor CS-TSSOS is able to deal with these problems since their characteristic functions are of degree $3$, forcing the relaxation order to be at least $2$. This again causes a memory leak.

We apply Algorithm~\ref{alg:Fast_Lower_Bounds} to characteristic functions of these MAX-3SAT problems with the following two sets of parameters:
\begin{align*}
d&=1, \quad  k=\left|\operatorname{supp}(f)\right|, \quad &&l=0,  \quad m = \text{sum of weights}. \\
d&=2, \quad k=1.5 \left|\operatorname{supp}(f)\right|,\quad &&l=0, \quad  m = \text{sum of weights}.
\end{align*}

The goal of this experiment is to exhibit the flexibility of Algorithm~\ref{alg:Fast_Lower_Bounds}. Namely, the trade-off between the quality of the lower bound and the time cost can be controlled freely by parameters $d$ and $k$. As a comparison, such a trade-off in other SOS-based algorithms such as TSSOS and CS-TSSOS is controlled by the order of relaxations, which may easily result in a memory leak (cf.~\ref{subsubsec:MAX-2SAT}). Experimental results can be found in Table \ref{Tab:Max-3SAT}, where the meaning of labels are the same as those in Table~\ref{Tab:Max-2SAT}.
\begin{table}[htbp]
\begin{tabular}{|c|c|c|c|c|c|c|}
  \hline
 \multirow{2}{*}{No}&  \multirow{2}{*}{min}& \multicolumn{2}{|c|}{$d=1$, $k=\left|\operatorname{supp}(f)\right|$} & \multicolumn{2}{|c|}{$d=2$, $k=3\left|\operatorname{supp}(f)\right|/2$} \\ \cline{3-6}
   & &bound & time &bound & time \\ \hline
1&12&6.2853&615&8.8433&2265\\ \hline
2&19&10.7006&582&13.6655&2126\\ \hline
3&5&4.2047&601&4.8462&2176\\ \hline
4&20&12.1123&584&15.1769&2098\\ \hline
5&14&7.4614&605&9.7210&2183\\ \hline
6&12&8.0066&595&10.0030&2131\\ \hline
7&17&10.1713&605&12.7487&2157\\ \hline
8&7&5.4025&635&6.4714&2325\\ \hline
9&17&10.6177&597&13.6209&2121\\ \hline
10&12&8.2999&573&10.6198&2108\\ \hline
\end{tabular}
\caption{Weighted MAX-3SAT problems} \label{Tab:Max-3SAT}
\end{table}
\subsection{rounding for MAX-2SAT problems}\label{subset:second experiment}
 Let $S_k$ be the support selected by  Algorithm \ref{alg:selection} with input $l=0$, $d=2$, $m$ being the {number of clauses}, and   $k$ being selected   such that the cardinality of $S=S_k \cup M_1$ is  $|M_p|$, where $M_1$  is the set of monomials with degree at most $1$, and  $M_p$ is the monomial basis containing $M_1$ and monomials $z_i z_j$ whenever logic variables $x_i$ and $x_j$ appear in the same clause, see \cite[Definition 1]{VVH2008}.
We compare  the rounding techniques presented in Section~\ref{sec:rounding} with the rounding techniques with   scaling factors $\rho_i^N$ and $2^{-(i-1)}$ in \cite{VVH2008} on randomly generated  MAX-2SAT problems. For the rounding method based on the Gram matrix,  the maximum number of iterations of SDPNAL+ is set to be $|S|$. For the rounding method based on the moment matrix, the maximum number of iterations of SDPNAL+ is set to be $\max(300,|S|)$.
\subsubsection{experiment on random problems}
 In this experiment, we conduct the experiment in  \cite[Table~7]{VVH2008}. For $n=25,30,35,40$ and $m=3n, 5 n, 7 n$, we randomly generate  $100$ MAX-2SAT instances with $n$ variables and $m$ clauses, compare our  rounding techniques with those given in \cite{VVH2008}, and record the frequencies of finding the optimum by each method. Results are presented in Table~\ref{tab:MAX-2SAT_rounding}, in which "Vars" denotes the number of variables, "clause" denotes the number of clauses, "Gram" and "Moment" are  the the  frequencies of finding the optimum by methods based on the Gram matrix and  the moment matrix repectively, $\rho_i^N$ and $2^{-(i-1)}$ are the  frequencies by recommended methods presented in \cite{VVH2008}. Table~\ref{tab:MAX-2SAT_rounding} shows that the rounding method based on lower rank moment approximations  has the highest frequency to find the optimum.

%
%
\begin{table}[htbp]
\begin{center}
\begin{tabular}{|c|c|c|c|c|c|}
\hline
Vars&Clauses&Gram&Moment&$\rho_i^N$&$2^{-(i-1)}$ \\ \hline
25&75&63&\textbf{96}&70&76\\ \hline
25&125&75&\textbf{93}&67&81\\ \hline
25&175&80&\textbf{91}&64&87\\ \hline
30&90&71&\textbf{95}&67&68\\ \hline
30&150&70&\textbf{91}&58&70\\ \hline
30&210&68&\textbf{87}&63&73\\ \hline
35&105&74&\textbf{93}&56&65\\ \hline
35&175&77&\textbf{91}&55&78\\ \hline
35&245&68&\textbf{87}&52&72\\ \hline
40&120&79&\textbf{92}&51&56\\ \hline
40&200&72&\textbf{87}&44&59\\ \hline
40&280&67&\textbf{83}&54&76\\ \hline
\end{tabular}
\end{center}
\caption{rounding for random MAX-2SAT problems}
\label{tab:MAX-2SAT_rounding}
\end{table}
\subsubsection{experiment on benchmark problems}
We test the rounding techniques respectively presented in Subsection~\ref{sec:rounding} and \cite{VVH2008} on randomly generated unweighted MAX-2SAT problems in 2016
MAX-SAT competition. Table~\ref{tab:MAX-2SAT_rounding2} indicates that   both of our rounding  methods outperform the method presented in  \cite{VVH2008}.

Furthermore, we remark that since the interior point method is used in \cite{VVH2008} to solve SDP problems, their rounding techniques  usually take  at least   $6000$ seconds. In comparison, our methods take less than $1000$ seconds, as our methods allow one to use SDPNAL+ to solve SDP problems.
\begin{table}[htbp]
\begin{center}
\begin{tabular}{|c|c|c|c|c|c|c|}
\hline
No&clause&min&Gram&Moment& $\rho_i^N$&$2^{-(i-1)}$  \\ \hline
1&1200&161&\textbf{162}&\textbf{162}&225&227\\ \hline
2&1200&159&\textbf{159}&164&215&194\\ \hline
3&1200&160&\textbf{160}&\textbf{160}&162&\textbf{160}\\ \hline
4&1300&180&\textbf{180}&185&226&243\\ \hline
5&1300&172&\textbf{173}&178&225&230\\ \hline
6&1300&173&\textbf{173}&174&245&253\\ \hline
7&1400&197&\textbf{198}&202&234&270\\ \hline
8&1400&191&\textbf{192}&199&255&246\\ \hline
9&1400&189&\textbf{189}&\textbf{189}&227&231\\ \hline
\end{tabular}
\end{center}
\caption{rounding on MAX-2SAT benchmarks}
\label{tab:MAX-2SAT_rounding2}
\end{table}

\bibliographystyle{ACM-Reference-Format}
\bibliography{max-sat-arxiv}

\end{document}